\documentclass[12pt, reqno]{amsart}
\usepackage{graphicx,amsmath,amsfonts,amssymb,amsthm,mathrsfs}

\newtheorem{theorem}{Theorem}[section]

\newtheorem{proposition}[theorem]{Proposition}

\theoremstyle{definition}
\newtheorem{definition}[theorem]{Definition}

\newtheorem{remark}[theorem]{Remark}
\numberwithin{equation}{section}

\begin{document}
\setcounter{page}{1}

\title{The $a$-number of Certain Hyperelliptic Curve}


\author{Vahid Nourozi, Farhad Rahmati and Saeed Tafazolian}
\address{Faculty of Mathematics and Computer Science, Amirkabir University of Technology\\
(Tehran Polytechnic), 424 Hafez Ave., Tehran 15914, Iran}
\email{nourozi.v@gmail.com; nourozi@aut.ac.ir}
\email{frahmati@aut.ac.ir}
\address{IMECC/UNICAMP, R. Sergio Buarque de Holanda, 651, Cidade Universitaria\\
 “Zeferino  Vaz”, 13083-859, Campinas, SP, Brazil}
\email{tafazolian@ime.unicamp.br}
\curraddr{}


\begin{abstract}
 In this paper, we compute a formula for the $a$-number of certain hyperelliptic curves given by the equation $y^2=x^{m}+1$  for infinitely many values of $m$. The same problem is studied for the curve corresponding to $y^2=x^{m}+x$.
\end{abstract}

\maketitle
\section{Introduction}

Let $k$ be an algebraically closed field of characteristic $p>0$. Let $\mathcal{A}$ be an abelian variety defiend over  $k$. Let $\alpha_p$ be the group scheme $\mbox{Spec}(k[X]/(X^p))$ with co-multiplication given by
$$X\rightarrow 1 \otimes X+X \otimes 1.$$
The group $\mbox{Hom}(\alpha_p, A)$ can be considered as $k$-vector space since $\mbox{End}(\alpha_p)=k$.
The  $a$-number $a(\mathcal{A} )$ defined to be the dimension of the vector space $\mbox{Hom}(\alpha_p, A)$.

Let $\mathcal{X}$ be a (non-singular, projective, geometrically irreducible, algebraic)
curve defined over $k$. One can define  the $a$-number $a(\mathcal{X})$ of $\mathcal{X}$ as  the $a$-number of its Jacobian variety $\mathcal{J}_{\mathcal{X}}$. As a matter of fact, the $a$-number of a curve is a  birational invariant which can defined as  the dimension of the space of exact holomorphic differentials.

The $a$-number of Hermitian curves computed by Gross in \cite{10}, and for Fermat and Hurwitz curves computed by Maria \cite{maria}. A few results on the rank of the Carteir operator (especially $a$-number) of curves introduced by Kodama and Washio \cite{13}, González \cite{8}, Pries and Weir \cite{17} and Yui \cite{Yui}.

 In this work, we consider the hyperelliptic curve $\mathcal{X}$ given by the equation $$y^2=x^{m}+1~\mbox{or} ~ y^2=x^{m}+x$$ over $k$.

 These families of hyperelliptic curves have been investigated for several reasons by many authors (see \cite{KW}, \cite{JPAA}, \cite{FF2}, \cite{V}).
  Here we are going to determine the  $a$-number $a(\mathcal{X})$ of $\mathcal{X}$ for infinitely many values of $m$.
See Theorem \ref{vahid}, \ref{vahid1} and \ref{vahid2}.

\section{The Cartier operator}
Let $k$ be an algebraically closed field of characteristic $p>0$. 
Let $\mathcal{X}$ be a 
curve defined over $k$.
The Cartier operator is a $p$-linear operator acting on the sheaf $\Omega^1_{\mathcal{X}}$ of differential
forms on $\mathcal{X}$ in positive characteristic.

 Let $K=k(\mathcal{X})$ be the function field of a curve $\mathcal{X}$ of genus $g$ defined over  $k$. A separating variable for $K$ is an element $x \in K \setminus K^p$.

\begin{definition}
  (The Cartier operator). Let $\omega \in  \Omega_{K/k}$. There exists $f_0,\cdots, f_{p-1}$ such
that $\omega = (f^p_0 + f^p_1x +\cdots + f^p_{p-1}x^{p-1})dx$. The Cartier operator $\mathscr{C}$ is defined by
$$\mathscr{C}(\omega) := f_{p-1}dx.$$
The definition does not depend on the choice of $x$ (see [\cite{100}, Proposition 1]).
\end{definition}

We refer the reader to [\cite{30}, \cite{40},\cite{100}, \cite{150}] for the proofs of the following statements.

\begin{proposition}
  (Global Properties of $\mathscr{C}$). For all $\omega \in  \Omega_{K/k}$ and all $f \in F$,
  \begin{itemize}
    \item [1.] $\mathscr{C}(f^p\omega) = f\mathscr{C}(\omega)$;
    \item [2.] $\mathscr{C}(\omega) = 0 \Leftrightarrow \exists h \in K, \omega = dh$;
    \item [3.] $\mathscr{C}(\omega) = \omega \Leftrightarrow \exists h \in K, \omega = dh/h$.
  \end{itemize}
\end{proposition}

\begin{remark}\label{r3.1}
\rm{ Moreover, one can easily show
that

\[\mathscr{C}(x^j dx) = \left\{
\begin{array}{ccc}
0 & \mbox{if}&  \hspace{.4cm} p \nmid j+1 \\
x^{s-1}dx & \mbox{if} &  \hspace{.4cm} j+1=ps. \\
\end{array}\right.\]}
\end{remark}

If $div(\omega)$ is effective then differential $\omega$ is holomorphic. The set $H^0(\mathcal{X}, \Omega^1)$ of holomorphic differentials is a $g$-dimensional $k$-vector subspace of $\Omega^1$ such that $\mathscr{C}(H^0(\mathcal{X}, \Omega^1)) \subseteq H^0(\mathcal{X}, \Omega^1)$. If $\mathcal{X}$ is a curve, then the $a$-number of $\mathcal{X}$ equals the dimension of the kernel of the Cartier operator $H^0(\mathcal{X}, \Omega^1)$ (or equivalently, the dimension of the space of exact holomorphic differentials on $\mathcal{X}$) (see \cite[5.2.8]{13}).

The Cartier operator and Hasse-Witt-matrix are dual to each other under the duality given by the Riemann-Roch theorem.
Let $\mathcal{B}=\{\omega_1,\cdots,\omega_g\}$ be a basis of the $k$-module of holomorphic differentials in $H^0(\mathcal{X}, \Omega^1)$. Then the representation matrix $M$ over $k$ of $\mathscr{C}$ with respect to this basis is called the Hasse-Witt matrix.

\vspace{.2cm}

Let $k$ be a field of characteristic $p>2.$ Let
$\mathcal{X}$ be a projective nonsingular hyperelliptic curve over
$k$ of genus $g$. Then $\mathcal{X}$ can be defined by an affine
equation of the form
$$y^2=f(x)$$
where $f(x)$ is a polynomial over $k$ of degree $d=2g+1$ or $d=2g+2$ without
multiple roots.

The differential 1-forms of the first kind on $\mathcal{X}$ form a
$k-$vector space $H^{0}(\mathcal{X},\Omega^{1})$ of dimension $g$
with basis
$$\mathcal{B}=\{\omega_i=\frac{x^{i-1}dx}{y}, \hspace{.2cm}
i=1,\hdots,g \}.$$
 The images under the operator $\mathscr{C}$
 are determined in the following way (see \cite{Yui}). Rewrite
 $$\omega_i=\frac{x^{i-1}dx}{y}=x^{i-1}y^{-p}y^{p-1}dx= y^{-p}x^{i-1} \sum_{j=0}^{N}
  c_j x^j dx,$$
 where the coefficients $c_j \in k$ are obtained from the
 expansion
 $$y^{p-1}=f(x)^{(q-1)/2}= \sum_{j=0}^{N} c_j x^j \hspace{.7cm} ~\mbox{with}~ N=\frac{p-1}{2}(d).$$
 Then we get for $i=1,\hdots,g,$

\begin{equation*}
\begin{aligned}
\omega_i = & y^{-p}(\sum_{\substack{j \\ i+j
 \neq 0~ mod~p }} c_j x^{i+j-1}dx)+ \sum_{l} c_{(l+1)p-i}
 ~\frac{x^{(l+1)p}}{y^p}\frac{dx}{x}.
\end{aligned}
\end{equation*}
Note here that $0 \leq l \leq \frac{N+i}{p}-1 < g -\frac{1}{2}. $
On the other hand, we know from Remark \ref{r3.1} that if
$\mathscr{C}(x^{r-1}dx) \neq 0$ then $r \equiv 0$ (mod $p$).
Thus we have
$$\mathscr{C}(\omega_i)=\sum_{l=0}^{g-1} {(c_{(l+1)p-i})}^{1/p}~ . \frac{x^{l}}{y}dx.$$
If we write $\omega=(\omega_1,\hdots,\omega_g)$ as a row vector we
have

$$\mathscr{C}(\omega)=  M(\mathcal{X})^{(1/p)}\omega,$$
where $M(\mathcal{X})$ is the $(g \times g)$ matrix with elements in $k$ given
as

\[M(\mathcal{X})=\left( \begin{array}{cccc}
c_{q-1}& c_{p-2} & \hdots & c_{p-g} \\
c_{2p-1}&c_{2p-2} & \hdots & c_{2p-g}\\
\vdots & \hdots & \hdots & \vdots \\
c_{gp-1}& c_{gp-2} & \hdots & c_{gp-g}
\end{array} \right).\]
\newline

\section{The $a$-number of Hyperelliptic Curve $y^2=x^m +1$}

In this section, we consider the hyperelliptic curve $\mathcal{X}$ given by the equation  $y^2=x^{m}+1$ over $k$.
This curve is of  genus $g=(m-1)/2$ (resp. $g=(m-2)/2$) if $m$ is odd (resp. $m$ is even).

Let 
 $\mathcal{B}=\{\omega_i=\frac{x^{i-1}dx}{y}, \hspace{.2cm}
i=1,\hdots,g \}$  be a basis for  
the differential 1-forms of the first kind on $\mathcal{X}$. Then 
 the rank of the Cartier operator $\mathscr{C}$ on the curve $\mathcal{X}$ equals the number of $i$ with $i\leq g$ such that
\begin{equation*}
 \begin{array}{ccc}
              \mathscr{C}(w_i) & = & \frac{1}{y}\mathscr{C}(x^{i-1}y^{p-1}dx) \\
               & = & \frac{1}{y}\mathscr{C}((x^{m}+1)^{\frac{p-1}{2}}x^{i-1}dx) \\
               & = &\frac{1}{y}\mathscr{C}(\sum_{j=0}^{\frac{p-1}{2}} a_j x^{j+i-1}dx) \neq 0,\\
             \end{array}
  \end{equation*}
  where $(x^{m}+1)^{\frac{p-1}{2}} = \sum_{j=0}^{\frac{p-1}{2}} a_j x^{jm}.$
From this we must have the equation of congruences mod $p$,
\begin{equation}\label{vvaa1}
  i+mj-1 \equiv p-1
\end{equation}
for some $0\leq j \leq \frac{(p-1)}{2}$. Equivalently,  the following equation 
\begin{equation}\label{12}
m(p-1-h) + i-1 \equiv p-1
\end{equation}
has a solution $h$ for $0 \leq h \leq \dfrac{p - 1}{2}$.

For the rest of this section, $M_m:=M(\mathcal{X})$ is the matrix representing the $p$-th power of the Cartier operator $\mathscr{C}$ on the curve $\mathcal{X}$ with respect to the basis $\mathcal{B}$.

\begin{theorem}\label{vahid}
  Let $\mathcal{X}$ be a hyperelliptic curve 
  given by the equation $y^2=x^m+1$. Suppose that $m=sp+1$, then 
  \begin{itemize}
    \item [1.] If $s = 2k+1$ and $k \geq 0$, then the $a$-number of the curve $\mathcal{X}$ equals
  $$\dfrac{(k+1)(p-1)}{2}.$$
    \item [2.] If $s = 2k$ and $k \geq 1$, then the $a$-number of the curve $\mathcal{X}$ equals
  $$\dfrac{k(p-1)}{2}.$$
  \end{itemize}
\end{theorem}

\begin{proof}
\begin{itemize}
   \item [(1).] At the first, if $m=sp+1=(2k+1)p+1$ with $k \geq$, then we prove that $\mbox{rank}(M_{m}) =\dfrac{k(p+1)}{2}$.
  
  In this case, $i \leq g$ and Equation (\ref{12}) mod $p$ reads
  \begin{equation}\label{1nn4}
              i - h - 1 \equiv 0
\end{equation}
In particular, if $k=0$ then $m= p+1$, where $i \leq g$ and Equation (\ref{1nn4}) be transformed into
  \begin{equation}\label{ih1}
              i \equiv h+1
\end{equation}
Take $l \in \mathbb{Z}_0^+$ so that $i= lp+h+1$, then $1 \leq lp+h+1 \leq \dfrac {p-1}{2}$. This implies that  $h\geq 0$ and $h< -3/2$, a contradictions. Thus, $rank(M_{p+1}) =0$.

If $k=1$ then $m=3p+1$, in this case we have $\dfrac{p}{2} \leq i \leq \dfrac{3p-1}{2}$. We need to find the solutions $h$ mode $p$ of the Equation (\ref{ih1}). Then
$$ \dfrac{p}{2} \leq lp+h+1 \leq \dfrac{3p-1}{2}.$$
As $h+1 \geq 0$ we obtain 
  \begin{equation*}
\Bigg\{
             \begin{array}{c}
              l\geq 0\\
              l < 3/2 \\
             \end{array}
\end{equation*}
Thus, we have two choices for $l$, i.e, $l=0$ or $l=1$. From this we have $\frac{1}{2}(p+1)$ choices for $h$, and so we conclude $\mbox{rank} (M_{3p+1})=\frac{1}{2}(p+1)$.

For $k\geq 2$, and $m=sp+1$ we can say $\mbox{rank}(M_{(2k+1)p+1})$ equals \par
 $\mbox{rank}(M_{(2k-1)p+1})$ plus the number of $i$ such that there is $h$ solution of the equation mod $p$
  \begin{equation*}
              i \equiv h+1
\end{equation*}
with $\dfrac{(2k-1)p}{2} \leq i \leq \dfrac{(2k+1)p-1}{2}$. Then
$$\dfrac{(2k-1)p}{2} \leq lp+h+1 \leq \dfrac{(2k+1)p-1}{2}.$$
This implies that
  \begin{equation*}
\Bigg\{
             \begin{array}{c}
              l\geq \dfrac{2k-1}{2}\\
              l < \dfrac{2k+1}{2} \\
             \end{array}
\end{equation*}
or equivalently we obtain $k=l$. In this case we have $\frac{1}{2}(p+1)$ choices for $h$. Therefore we get 
$$\mbox{rank}(M_{(2k+1)p+1})= \mbox{rank}(M_{(2k-1)p+1})+ \frac{1}{2}(p+1).$$
Now the our claim on the rank of $M_{(2k+1)p+1}$ follows by induction on $k$.

Then $a(\mathcal{X}_{(2k+1)p+1})= \dfrac{(k+1)(p-1)}{2}$ can be computed from $$a(\mathcal{X}_{(2k+1)p+1}) = g(\mathcal{X}_{(2k+1)p+1}) - rank (M_{(2k+1)p+1})$$.

 \item [(2.)] At first we cliam that $rank(M_{sp+1}) =\dfrac{k(p+1)}{2}$, with $m=2kp+1$ and $k \geq 1$.
In this case, $i \leq g$ and Equation \ref{12} mod $p$ reads
  \begin{equation}\label{154}
              i - h - 1 \equiv 0
\end{equation}
In particular, if $k=1$ then $m= 2p+1$, where $i \leq g$ and Equation \ref{154} be transformed into
  \begin{equation}\label{ih2}
              i \equiv h+1
\end{equation}
Take $l \in \mathbb{Z}_0^+$ so that $i= lp+h+1$, then $1 \leq lp+h+1 \leq p $. Thus, we have one choices for $l$. From this we have $\frac{1}{2}(p+1)$ choices for $h$, and yielding $rank(M_{2p+1})=\frac{1}{2}(p+1)$.

If $k = 2$, then $m=4p+1$, in this case we have $1 \leq i \leq 2p$. We need to find the solutions $h$ mode $p$ of the above Equation \ref{ih2}. Then
$$ 1\leq lp+h+1 \leq 2p.$$
As $h+1 \geq 0$
  \begin{equation*}
\Bigg\{
             \begin{array}{c}
              l\geq 0\\
              l < 2 \\
             \end{array}
\end{equation*}
Thus, we have two choices for $l$, i.e, $l=0$ or $l=1$. From this we have $(p+1)$ choices for $h$, and yielding $rank(M_{4p+1})=(p+1)$.

For $k\geq 3$, and $m=sp+1$ we can say $rank(M_{2kp+1})$ equals $rank(M_{2(k-1)p+1})$ plus the number of $i$ such that there is $h$ solution of the equation mod $p$
  \begin{equation*}
              i \equiv h+1
\end{equation*}
with $1 \leq i \leq 2kp$. Then
$$(2k-2)p \leq lp+h+1 \leq 2kp.$$
Hence,
  \begin{equation*}
  l=2k
\end{equation*}
In this case we have $\frac{1}{2}(p+1)$ choices for $h$. This implies that
$$rank(M_{2kp+1})= rank(M_{(2(k-1)p+1})+ \frac{1}{2}(p+1).$$

Now our claim on the rank of $M_{2kp+1}$ follows by induction on $k$.

Then $a(\mathcal{X}_{2kp+1})= \dfrac{(k)(p-1)}{2}$ can be computed from $$a(\mathcal{X}_{2kp+1}) = g(\mathcal{X}_{2kp+1}) - rank (M_{2kp+1})$$.
\end{itemize}
\end{proof}

\begin{theorem}\label{vahid1}
Suppose that $m=sp-1$ then,
\begin{itemize}
  \item [1.] If $s = 2k+1$ and $k \geq 0$, then the $a$-number of the curve $\mathcal{X}$ equals
  $$\dfrac{k(p-1)}{2}.$$
  \item [2.] If $s = 2k$ and $k \geq 1$, then the $a$-number of the curve $\mathcal{X}$ equals
  $$\dfrac{k(p-1)}{2}.$$
\end{itemize}
\end{theorem}

\begin{proof}

   Proof of this theorem is similar to Theorem \ref{vahid}.

\end{proof}

\section{The $a$-number of Hyperelliptic Curve $y^2=x^m +x$}
In this section, we consider the hyperelliptic curve $\mathcal{X}$ given by the equation  $y^2=x^{m}+x$ over $k$.
This curve is of  genus $g=(m-1)/2$ (resp. $g=(m-2)/2$) if $m$ is odd (resp. $m$ is even).

Let 
 $\mathcal{B}=\{\omega_i=\frac{x^{i-1}dx}{y}, \hspace{.2cm}
i=1,\hdots,g \}$  be a basis for  
the differential 1-forms of the first kind on $\mathcal{X}$. Then 
 the rank of the Cartier operator $\mathscr{C}$ on the curve $\mathcal{X}$ equals the number of $i$ with $i\leq g$ such that
\begin{equation*}
             \begin{array}{ccc}
              \mathscr{C}(w_i) & = & \frac{1}{y}\mathscr{C}(x^{i-1}y^{p-1}dx) \\
               & = & \frac{1}{y}\mathscr{C}(x^{\frac{p-1}{2}}(x^{m-1}+1)^{\frac{p-1}{2}}x^{i-1}dx) \\
               & = &\frac{1}{y}\mathscr{C}(\sum_{j=0}^{\frac{p-1}{2}} a_j x^{j+i-1}dx) \neq 0\\
             \end{array}
  \end{equation*}
    where $(x^{m-1}+1)^{\frac{p-1}{2}} = \sum_{j=0}^{\frac{p-1}{2}} a_j x^{j(m-1)}.$
From this we must have the equation of congruences mod $p$,
\begin{equation}\label{vvaa}
  i+(m-1)j-1 \equiv p-1
\end{equation}
for some $0\leq j \leq \frac{(p-1)}{2}$. Equivalently,  the following equation 
\begin{equation}\label{14}
m(p-1-h) +t+ i-1 \equiv p-1
\end{equation}
has a solution $h$ for $0 \leq t \leq h \leq \dfrac{p - 1}{2}$.





\begin{theorem}\label{vahid2}
  If $m=sp$ for $s = 2k+1$ and $k \geq 0$, then the $a$-number of the curve $\mathcal{X}$ equals
  $$\dfrac{(k+1)(p-1)}{2}.$$
\end{theorem}
\begin{proof}
 At first we cliam that $rank(M_{sp}) =\dfrac{k(p+1)}{2}$, with $m=(2k+1)p$ and $k \geq 0$.

In this case, $i \leq g$ and Equation \ref{14} mod $p$ reads
  \begin{equation}\label{155}
              i +t \equiv 0
\end{equation}
Peculiarly, if $k=0$ then $m= p$, where $i \leq g$ and Equation \ref{155} be transformed into
  \begin{equation}\label{ih6}
              i \equiv -t
\end{equation}
Take $l \in \mathbb{Z}_0^+$ so that $i= lp-t$, then $1 \leq lp-t \leq \dfrac {p}{2}$. From this $t\geq -1$ and $t\geq 0$, a contradictions. Thus, $rank(M_{p}) =0$.

If $k=1$ then $m=3p$, in this case we have $\dfrac{p}{2} \leq i \leq \dfrac{3p}{2}$. We need to find the solutions $h$ mode $p$ of the above Equation \ref{ih6}. Then
$$ \dfrac{p}{2} \leq lp-t \leq \dfrac{3p}{2}.$$
As $t \geq 0$
  \begin{equation*}
\Bigg\{
             \begin{array}{c}
              l\geq 0\\
              l < 3/2 \\
             \end{array}
\end{equation*}
Thus, we have two choices for $l$, i.e, $l=0$ or $l=1$. From this we have $\frac{1}{2}(p+1)$ choices for $t$, and yielding $rank(M_{3p})=\frac{1}{2}(p+1)$.

For $k\geq 2$, and $m=sp$ we can say $rank(M_{(2k+1)p})$ equals $rank(M_{(2k-1)p})$ plus the number of $i$ such that there is $t$ solution of the equation mod $p$
  \begin{equation*}
              i \equiv -t
\end{equation*}
with $\dfrac{(2k-1)p}{2} \leq i \leq \dfrac{(2k+1)p}{2}$. Then
$$\dfrac{(2k-1)p}{2} \leq lp-t \leq \dfrac{(2k+1)p}{2}.$$
Hence,
  \begin{equation*}
  l=k
\end{equation*}
In this case we have $\frac{1}{2}(p+1)$ choices for $t$. This implies that
$$\mbox{rank}(M_{(2k+1)p})= \mbox{rank}(M_{(2k-1)p})+ \frac{1}{2}(p+1).$$

Now our claim on the rank of $M_{(2k+1)p}$ follows by induction on $k$.

Then $a(\mathcal{X}_{(2k+1)p})= \dfrac{(k+1)(p-1)}{2}$ can be computed from $$a(\mathcal{X}_{(2k+1)p}) = g(\mathcal{X}_{(2k+1)p}) - \mbox{rank} (M_{(2k+1)p}).$$

\end{proof}

\section*{Acknowledgement}
{The third author was supported by
FAPESP/SP-Brazil grant 2017/19190-5.}

\bibliographystyle{amsplain}

\end{document}